\documentclass[a4paper]{amsart}
\usepackage[backend=biber, isbn=false, doi=false, url=false, style=numeric, sortcites=true, citestyle=numeric, giveninits=true]{biblatex}
\bibliography{../Style/references.bib}

\usepackage[utf8]{inputenc}
\usepackage[english]{babel}
\usepackage{amsthm}
\usepackage{amsmath}
\usepackage{amsfonts}
\usepackage{mathtools}
\usepackage{enumitem}
\usepackage{xcolor}
\usepackage{amssymb}
\usepackage{csquotes}
\usepackage{stmaryrd}
\usepackage[capitalise, noabbrev]{cleveref}
\usepackage{cancel}
\setlist[enumerate,1]{label = (\roman*)}

\theoremstyle{definition}
\newtheorem{definition}{Definition}[section]
\newtheorem{proposition}[definition]{Proposition}
\newtheorem{theorem}[definition]{Theorem}
\newtheorem{lemma}[definition]{Lemma}
\newtheorem{corollary}[definition]{Corollary}
\newtheorem{example}[definition]{Example}
\newtheorem*{main-theorem}{Theorem}
\newtheorem*{fact}{Fact}
\theoremstyle{remark}

\renewcommand{\restriction}{ {\upharpoonright} }

\newcommand{\M}{\mathcal{M}}
\newcommand{\N}{\mathcal{N}}
\newcommand{\U}{\mathcal{U}}
\newcommand{\F}{\mathcal{F}}
\renewcommand{\S}{\mathcal{S}}
\renewcommand{\P}{\mathcal{P}}
\newcommand{\Q}{\mathcal{Q}}
\DeclareMathOperator{\dep}{=}

\DeclareMathOperator{\Pow}{\wp}
\DeclareMathOperator{\dom}{dom}

\DeclareMathOperator{\Fv}{Fv}

\newcommand\ivee{\mathrel{\rotatebox[origin=c]{270}{$\geqslant$}}}

\newcommand{\existsone}{\exists^1}
\newcommand{\forallone}{\forall^1}
\newcommand{\svee}{\vee_{\mathrm{s}}}
\newcommand{\sexists}{\exists_{\mathrm{s}}}
\newcommand{\unsim}{\mathord{\sim}}
\newcommand{\cneg}{\mathord{\dot{\sim}}}

\newcommand\fol{\mathsf{FO}}
\newcommand\indl{\mathsf{FO}( \perp_c)}
\newcommand\eso{\mathsf{ESO}}

\newcommand\fodl{\mathsf{FO}(\dep(\dots))}
\newcommand\foil{\mathsf{FO}(\perp_c)}

\newcommand\nmodels{\mathbin{\cancel{\models}}}

\title[Compactness in Team Semantics]{Compactness in Team Semantics}
\date{\today}
\author{Joni Puljujärvi}
\author{Davide Emilio Quadrellaro}
\email{joni.puljujarvi@helsinki.fi}
\email{davide.quadrellaro@helsinki.fi}
\address{Department of Mathematics and Statistics, University of Helsinki, P.O. Box 68 (Pietari Kalmin katu 5), 00014 Helsinki, Finland}
\thanks{J.~Puljujärvi was partially supported by the Academy of Finland, grant 322795. D. E. Quadrellaro was supported by grant 336283 of the Academy of Finland and Research Funds of the University of Helsinki.}
\keywords{team semantics, compactness,  dependence logic, ultraproducts}
\subjclass[2020]{03B60, 03C20, 03C85}

\begin{document}
\maketitle

\begin{abstract}
    We provide two proofs of the compactness theorem for extensions of first-order logic based on team semantics. First, we build upon Lück's~\cite{luck2020team} ultraproduct construction for team semantics and prove a suitable version of Łoś' Theorem. Second, we show that by working with suitably saturated models, we can generalize the proof of Kontinen and Yang~\cite{https://doi.org/10.48550/arxiv.1904.08695} to sets of formulas with arbitrarily many variables.
\end{abstract}

\section{Introduction}

In this article we provide two alternative proofs of a general compactness theorem for in\-de\-pend\-ence logic and related logics based on team semantics. In particular, we prove the following claim:

\begin{main-theorem}\label{main.theorem}
    Let $\Gamma$ be a set of formulas of independence logic. If every finite subset $\Gamma_0$ of $\Gamma$ is satisfiable, then $\Gamma$ is satisfiable.
\end{main-theorem}

Team semantics is a general semantic framework, originally introduced by \textcite{Hodges}, which allows to consider several logics extending first-order logic. In particular,  \textcite{Vaananen2007-VNNDLA} used team semantics to provide a new approach to independence-friendly logic, called \emph{dependence logic}. This was later extended to \textit{inclusion} and \textit{independence} logic, which were respectively introduced in \cite{galliani2012inclusion} and \cite{gradel2013dependence}. The interest in these logics arises from the fact that they are capable of expressing dependencies between variables. The Armstrong's axioms of functional dependencies are, for instance, sound in dependence logic.

Already in \cite{Vaananen2007-VNNDLA}, Väänänen proved several metatheoretical properties of dependence logic, and in particular he gave a translation from dependence logic to the existential fragment $\eso$ of second order logic. This translation was  later adapted in \cite{galliani2012inclusion} and \cite{gradel2013dependence} also to inclusion and independence logic, the latter of which proved to be equivalent to $\eso$. Using this translation, and the compactness of $\eso$, \textcite{Vaananen2007-VNNDLA} proved a version of the compactness theorem for sentences of dependence logic, and the same technique shows compactness for sentences of inclusion and independence logics. However, unlike in classical first-order logic, in logics over team semantics compactness for formulas does not follow straightforwardly from compactness for sentences. In fact, in the context of team semantics, variables represent a vector of values, and their semantic value cannot be mimicked by constants. Compactness for formulas has been considered only recently by \textcite{https://doi.org/10.48550/arxiv.1904.08695}, who proved a version of compactness for sets of formulas with countably many free variables.
\noindent 

The structure of the present article is the following. In \cref{preliminaries}, we recall the basic definition and terminology of dependence and independence logic and we introduce the problem of compactness for formulas in greater details.

In \cref{compactness-by-ultraproducts}, we provide a proof of Łoś' Theorem for several extensions of first-order logic by team semantics, resulting in a compactness proof which resembles the standard proof of compactness of first-order logic. Following \textcite{luck2020team}, we adapt the ultraproduct construction to the setting of team semantics and extend his proof of a suitable version of Łoś' Theorem for rather weak team-based logics to cover, among others, independence logic both with lax and strict semantics.

Differently, in \cref{compactness-by-translation} we show that, by working with suitably saturated models, the underlying idea of Kontinen and Yang's proof can be adapted to sets of formulas with arbitrarily many variables. The key ingredients of this second proof are thus the translation of formulas of independence logic into the existential fragment of second order logic and the use of models which are suitably saturated in the first-order sense and whose existence is guaranteed by basic model-theoretic results. We conclude the paper with a few suggestions of further research.

We would like to thank Tapani Hyttinen for suggesting the approach we take in \cref{compactness-by-translation} and for several useful remarks. We also thank Aleksi Anttila, Åsa Hirvonen, Jouko Väänänen and Fan Yang for useful discussions and pointers to the literature.

\section{Logics over Team Semantics}\label{preliminaries}

In this section, we recall the syntax, semantics and basic properties of (in)de\-pend\-ence logic and related formalisms. We refer the reader to \textcite{Vaananen2007-VNNDLA,galliani2012inclusion,gradel2013dependence} for a general introduction to logics based on team semantics.

\subsection{First-Order Logic in Team Semantics.}

The logics which we shall be dealing with are all extensions of first-order logics (which we will denote by $\fol$) by means of team semantics. As negation has a subtle role in team semantics, we shall work with first-order formulas in negation normal form and treat negated atomic formulas similarly to atomic formulas. Let $\tau$ be a first-order signature. Then $\tau$-formulas of $\fol$ are given by the syntax
\[
	\phi \Coloneqq t=t'  \mid \neg t=t' \mid R(\vec{t}) \mid \neg R(\vec{t}) \mid  \phi \land \phi \mid\phi\vee\phi  \mid \exists x \phi \mid \forall x \phi,	
\]
where $\vec{t}=(t_i)_{i< n},t,t'$ are $\tau$-terms and $R\in\tau$ is an $n$-ary relation symbol. 

We use teams to provide first-order formulas with a semantics. Let $D$ be a set of variables, and $\M$ a $\tau$-structure. An \emph{assignment} of $\M$ is a map $s\colon D\to \M$, while a \emph{team} of $\M$ is a set of assignments $X\subseteq \M^D$. The set $D$ is called the \emph{domain} of $X$ and denoted by $\dom(X)$. The interpretation of a $\tau$-term $t$ in a structure $\M$ with assignment $s$ is defined as usual, and, abusing notation, we denote the interpretation by $s(t)$. If $\vec{t} = (t_0,\dots,t_{n-1})$ is a tuple of terms, we denote by $s(\vec{t})$ the tuple $(s(t_0),\dots,s(t_{n-1}))$. The interpretation of a quantifier-free first-order formula in a structure $\M$ and team $X$ is defined as follows.
\begin{enumerate}
    \item $\mathcal{M}\models_X t=t'$ if $s(t)=s(t')$ for all $s\in X$.
    \item $\mathcal{M}\models_X \neg t=t'$ if $s(t)\neq s(t')$ for all $s\in X$.
    \item $\mathcal{M}\models_X R(\vec{t})$ if $s(\vec{t})\in R^{\M}$ for all $s\in X$.
    \item $\mathcal{M}\models_X \neg R(\vec{t})$ if $s(\vec{t})\notin R^{\M}$ for all $s\in X$.
    \item $\mathcal{M}\models_X \psi \land \chi$ if $\mathcal{M}\models_X \psi$ and $\mathcal{M}\models_X \chi$.
    \item $\mathcal{M}\models_X \psi \vee \chi$ if there are $Y,Z\subseteq X$ such that $ Y\cup Z = X$,  $\mathcal{M}\models_Y \psi$ and $\mathcal{M}\models_Z \chi$.
\end{enumerate}
Disjunction with the above semantics is often called the \emph{tensor disjunction}, as opposed to, for instance, intuitionistic disjunction (see below).

Interpreting quantifiers requires us to introduce two operations on teams, \emph{supplementation} and \emph{duplication}. Let $\M$ be a $\tau$-structure and $s$ an assignment of $\M$. Then, given a variable $x$ and an element $a\in\M$, we denote by $s(a/x)$ the assignment $s'$ such that $s'(x) = a$ and $s'(y) = s(y)$ for all $y\neq x$. Let then $X$ be a team of $\M$, and denote $\Pow^+(\M)\coloneqq \wp(\M)\setminus\{\emptyset\}$. Given a function $F\colon X\to \Pow^+(\M)$, we call
\[
    X(F/x) \coloneqq \{ s(a/x) \mid \text{$s\in X$ and $a\in F(s)$} \} 
\]
the \emph{supplemented} team of $X$ over $x$ by $F$, and the team
\[
    X(M/x) \coloneqq \{ s(a/x) \mid \text{$s\in X$ and $a\in \M$} \}
\]
the \emph{duplicate} team of $X$ over $x$. The interpretation of existential and universal quantifiers is then defined as follows.
\begin{enumerate}[resume]
    \item $\mathcal{M}\models_X \exists x \psi$ if there is a function $F\colon X\to \wp^+(\M)$ such that $ \mathcal{M}\models_{X(F/x)} \psi$.
    \item $\mathcal{M}\models_X\forall x \psi $ if $  \mathcal{M}\models_{X(M/x)} \psi$. 
\end{enumerate}

We remark that the semantics that we obtain in this way is equivalent to the usual Tarski semantics of first-order logic (or,  that first-order logic is \emph{flat}; see below) in the following sense: for any $\fol$-formula $\phi$, any model $\M$ and any team $X$, we have that $\M\models_X \phi$ if and only if $\M\models_s \phi$ for all assignments $s\in X$.

\subsection{Dependency Atoms.}

Despite being conservative over first-order formulas, team semantics allows us for a greater expressive power. In particular, it makes possible to add to the basic vocabulary of first-order logic a set of different dependency atoms, with the purpose of expressing several ways in which variables relate to each other. We shall be interested in the \emph{dependence atom} $\dep(\dots)$, the \emph{independence atom} $\perp_c$, the \emph{inclusion atom} $\subseteq$ and the \emph{exclusion atom} $|$, for their role in expressing dependencies over databases (e.g. functional dependency, inclusion dependency and embedded multivalued dependency). We redirect the reader to \cite{fagin1984theory} for a survey. Next we give the semantics of these atoms.

\begin{enumerate}[resume]
    \item $\mathcal{M}\models_X \dep(\vec{x},y)$ if for all $s,s'\in X$, if $s(\vec{x})=s'(\vec{x})$  then $s(y)=s'(y)$.
    \item $\mathcal{M}\models_X \vec{x}\subseteq \vec{y}$, where $\vec{x}$ and $\vec{y}$ are tuples of the same length, if for all $s\in X$ there is some  $s'\in X$ such that $s(\vec{x})=s'(y)$.
    \item $\mathcal{M}\models_X \vec{x}\perp_{\vec{z}}\vec{y}$ if for all $s,s'\in X$ such that $ s(\vec{z})=s'(\vec{z})$, there is $s''\in X$ such that $ s(\vec{x}\vec{z})=s''(\vec{x}\vec{z})$ and  $s'(\vec{y})=s''(\vec{y})$;
    \item $\mathcal{M}\models_X \vec{x}\mid\vec{y}$, where $\vec{x}$ and $\vec{y}$ are tuples of the same length, if for all $s,s'\in X$ we have $s(\vec{x})\neq s'(\vec{y})$.
\end{enumerate}
Note that we do not need to allow terms to occur in the above atoms, as an atom of the form $\dep(t_0,\dots,t_n)$ is equivalent to $\exists x_0\dots \exists x_n\; (\dep(x_0,\dots,x_n) \land \bigwedge_{i\leq n}x_i=t_i)$.

For any set of atoms $C\subseteq \{ \dep(\dots), \perp_c, \subseteq, | \}$, we write $\fol(C)$ for the extension of first-order logic obtained by closing the first-order atomic and negated atomic formulas and the atomic formulas using symbols of $C$ under conjunction, disjunction and existential and universal quantifiers (but not negation; negation only occurs in front of first-order atoms). In particular, we refer to $\fol(\dep(\dots))$ as \emph{dependence logic}, to $\fol(|)$ as \emph{exclusion logic}, to $\fol(\subseteq)$ as \emph{inclusion logic} and to $\fol(\perp_c)$ as \emph{independence logic}. More precisely, the syntax of e.g. independence logic is
\[
	\phi \Coloneqq t=t'  \mid \neg t=t' \mid R(\vec{t}) \mid \neg R(\vec{t}) \mid \vec{x}\perp_{\vec{z}}\vec{y} \mid  \phi \land \phi \mid\phi\vee\phi  \mid \exists x \phi \mid \forall x \phi.
\]

\noindent Notice that the previous list of dependence atoms is by no means exhaustive. In particular, a notion of generalized dependency atoms was introduced in \cite{kuusisto2015double} and further studied e.g. in \cite{galliani2016strongly}.

\subsection{Strict Semantics and Further Operations.}

Besides for atomic dependencies, team semantics also makes it possible to express a wide range of different connectives and operations. One source of variation with respect to the framework that we introduced above is that of so called \emph{strict semantics} (as opposed to the \emph{lax semantics}). This refers to the following alternative semantics for disjunction and existential quantification.
\begin{enumerate}[resume]
    \item $\mathcal{M}\models_X \psi \svee \chi $ if there are $Y,Z\subseteq X$  such that $ Y\cap Z=\emptyset$, $ Y\cup Z=X$,  $\mathcal{M}\models_Y \psi$  and $\mathcal{M}\models_Z \chi$.
    \item $\mathcal{M}\models_X \sexists x \psi$ if there is a supplement function $F\colon X\to \Pow^+(\M)$ such that $\mathcal{M}\models_{X(F/x)} \psi$ and,  for all $s\in X$, $F(s)=\{a\}$ for some $a\in \M$.
\end{enumerate}
We refer the reader to \cite{galliani2012inclusion} for several results on logic with strict semantics.

Some further connectives often considered in team semantics are the following ones.
\begin{enumerate}[resume]
    \item $\mathcal{M}\models_X \phi \ivee \psi$ if $\mathcal{M}\models_X \phi$ or $\mathcal{M}\models_X \psi$;
    \item $\mathcal{M}\models_X \phi \to \psi$ if for all $Y\subseteq X$, $\mathcal{M}\models_X \phi $ entails $\mathcal{M}\models_X \psi $;
    \item $\mathcal{M}\models_X \cneg \phi$ if $X=\emptyset$ or $\mathcal{M}\nmodels_X \phi $;
    \item $\mathcal{M}\models_X  \unsim\phi$ if $\mathcal{M}\nmodels_X \phi$.
\end{enumerate}

\noindent The operations $\ivee$ and $\to$ are respectively called \emph{intuitionistic disjunction} and \emph{intuitionistic implication} and were originally considered in \cite{zbMATH05593776}. The operations $\cneg$ and $\unsim$ are respectively called \emph{weak classical negation} and simply \emph{classical negation}, and were respectively introduced in \cite{YANG20191128} and \cite{Vaananen2007-VNNDLA}.

Finally, one further source of variation can be given by the quantifier clauses, as one can consider the quantifiers $\existsone$ and $\forallone$ from \cite{kontinen2009definability} whose semantics is the following.
\begin{enumerate}[resume]
    \item $\mathcal{M}\models_X \existsone x \psi$ if there is an element $a\in \M$ such that $ \mathcal{M}\models_{X(a/x)} \psi$;
    \item $\mathcal{M}\models_X \forallone x \psi$ if for any element $a\in \M$ we have $ \mathcal{M}\models_{X(a/x)} \psi$.
\end{enumerate}
where we denote by $X(a/x)$ the set $\{s(a/x) \mid s\in X\} $. 

We abide to the convention set above and we write $\fol(C)$ for any extension of first-order logic by a set of dependencies, connectives and quantifiers $C$. Of course, the list of operations that we provided is by no means exhaustive. As flexibility is one key feature of team semantics, several other choices are possible. We shall focus here on the present ones as they are the most studied in the literature. It will often be clear how to extend the methods of this article to other contexts.

\subsection{Basic Properties and Translation into $\eso$.}

We shall briefly recall some basic properties and facts about logics over team semantics.  We refer the reader to \cite{galliani2012inclusion,gradel2013dependence, Vaananen2007-VNNDLA} for the proofs. We focus on formulas of independence logic and related fragments as they are the most common ones.

If $V\subseteq\dom(X)$, we denote by $X\restriction V$ the team $\{s\restriction V \mid s\in X\}$.

\begin{fact}  The following hold.
	\begin{enumerate}
	    \item \emph{Locality}: For any model $\M$, any team $X\subseteq \M^D$ and any formula $\phi$ of $\fol( \dep(\dots), \perp_c, \subseteq, | )$, we have $\M\vDash_{X} \phi$ if and only if $\M\vDash_{X\restriction \Fv(\phi)} \phi$.
		\item \emph{Empty team property}: For any model $\M$ and any formula $\phi$ of $\fol( \dep(\dots), \perp_c, \subseteq, | )$ we have $\M\vDash_{\emptyset} \phi$.
		\item \emph{Downwards closure}: For any model $\M$, any team $X\subseteq Y \subseteq \M^D$ and any formula $\phi$ of $\fol( \dep(\dots), | )$, if $\M\vDash_{Y} \phi$, then $\M\vDash_{X} \phi$.
		\item \emph{Union-closure}: For any model $\M$, any $X, Y \subseteq \M^D$ and any formula $\phi$ of $\fol(\subseteq )$, if $\M\vDash_{X} \phi$ and $\M\vDash_{Y} \phi$, then $\M\vDash_{X\cup Y} \phi$.
		\item \emph{Flatness}: For any model $\M$, any team $X \subseteq \M^D$ and any formula $\phi$ of $\fol$, we have $\M\vDash_{X} \phi$ if and only if $\M\vDash_s \phi$ for all $s\in X$.
	\end{enumerate}    
\end{fact}

\noindent It is easy to show that flatness is equivalent to the combination of the empty team property, union-closure and downwards closure, thus showing the connection between the former properties.

We notice in particular that locality does not hold if one replaces disjunction and existential quantification by those defined via strict semantics. Given the empty team property, it is important that when we consider whether a formula, or a set of formulas, is satisfiable, we require the satisfying team to be nonempty. Allowing the empty team would render every formula satisfiable in most logics we are interested in.

\begin{definition}[Satisfiability]
    We say that a set $\Gamma$ of formulas is satisfiable if there is a structure $\M$ and a nonempty team $X$ of $\M$ such that $\M\models_X\Gamma$.
\end{definition}

The expressive power of logics over team semantics ranges from that of first order logic to full second-order logic. Before providing a precise characterisation of the expressive power of independence logic, we shall first recall several definitions related to team properties.

If $X$ is a team of $\M$ and $x_0,\dots,x_{n-1}\in\dom(X)$, we denote by $X[x_0,\dots,x_{n-1}]$ the relation $\{(s(x_0),\dots,s(x_{n-1})) \mid s\in X\}$.

\begin{definition}[Team Property] \label{team.property}\quad
    \begin{enumerate}
        \item We say that a class $\P$ of pairs $(\M,X)$, where $\M$ is a structure and $X$ is a team of $\M$, is a \emph{team property} with domain $D$ if
        \begin{enumerate}
            \item $(\M,X)\in\P \implies D\subseteq\dom(X)$ and
            \item $\P$ is closed under isomorphism, i.e. if $(\M, X)\in\P$ and $\pi\colon\M\to\N$ is an isomorphism, then $(\N,\pi(X))\in\P$, where $\pi(X) = \{ \pi\circ s \mid s\in X \}$.
        \end{enumerate}

        \item A team property $\P$ with domain $D$ is \emph{local} if for all structures $\M$ and teams $X$ of $\M$ with $D\subseteq\dom(X)$ we have
        \[
            (\M,X)\in\P \iff (\M,X\restriction D)\in\P.
        \]
        
        \item A team property $\P$ has the \emph{empty team property} if for all structures $\M$, $(\M,\emptyset)\in\P $.
        
        \item A team property $\P$ is \emph{downwards closed} if for all structures $\M$ and teams $Y\subseteq X$ of $\M$,
        \[
            (\M,X)\in\P \implies (\M,Y) \in\P.
        \]
        
        \item A team property $\P$ is \emph{union-closed} if for all structures $\M$ and teams $X,Y$ of $\M$,
        \[
            (\M,X)\in\P\ \text{and}\ (\M,Y)\in\P \implies (\M,X\cup Y) \in\P.
        \]
        
        \item A team property $\P$ is \emph{flat} if for all structures $\M$ and teams $X$ of $\M$,
        \[
            (\M,X)\in\P \iff \text{$(\M,\{s\})\in\P$ for all $s\in X$.}
        \]

        \item A team property $\P$ with domain $\{x_0,\dots,x_{n-1}\}$ is \emph{first-order} (or \emph{elementary}) if there is a first-order $\tau_\P\cup\{R\}$-sentence $\phi(R)$ such that
        \[
            \P = \{(\M,X) \mid (\M, X[x_0,\dots,x_{n-1}])\models\phi(R)\},
        \]
        where $\tau_\P = \bigcap_{(\M,X)\in\P}\tau_\M$ and $\tau_\M$ is the vocabulary of $\M$.
        
        \item A team property $\P$ with domain $\{x_0,\dots,x_{n-1}\}$ is \emph{existential second-order} if there is an existential second-order $\tau_\P\cup\{R\}$-sentence $\phi(R)$, such that
        \[
            \P = \{(\M,X) \mid (\M, X[x_0,\dots,x_{n-1}])\models\phi(R)\}.
        \]
    \end{enumerate}
\end{definition}
A formula $\phi$ of a logic $\mathcal{L}$ based on team semantics gives rise to a team property $\llbracket\phi\rrbracket_{\mathcal{L}}$, whose domain consists of the free variables of $\phi$, in the obvious way:
\[
    (\M,X)\in\llbracket\phi\rrbracket_{\mathcal{L}} \iff \M\models_X\phi.
\]

The notion of a team property allows us to characterise the expressive power of logics over team semantics. Most specifically, the logics which we are most interested in, i.e. dependence logic, inclusion logic and independence logic, can be all seen as fragments of the existential fragment of second order logic ($\eso$). In particular, it was shown in \cite{gradel2013dependence, galliani2012inclusion} that independence logic is expressively complete with respect to all team properties which can be defined in $\eso$.  The following theorem makes this statement more precise. 

\begin{theorem}[Translation to $\eso$] \label{translation} Independence logic is expressively equivalent to existential second-order logic in the following sense.
    	\begin{enumerate}
            \item \label{translation1}  Let $\phi(v_0,\dots,v_{n-1})$ be a $\tau$-formula of independence logic, then there is a $\tau\cup\{R\}$-formula $\chi(R)$ of $\eso$, where $R$ is a fresh $n$-ary predicate symbol, such that
            \[
            \M\models_X \phi(v_0,\dots,v_{n-1}) \iff (\M,X[\vec{v}])\models \chi(R).
            \]
            \item Let $\chi(R)$ be an $\tau\cup \{R\}$-sentence of $\eso$, there is a $\tau$-formula $\phi(v_0,\dots,v_{n-1})$ of independence logic such that
            \[
            \M\models_X  \phi \iff  (\M, X[\vec{v}])\models \exists x R(x)\to \chi(R).
            \]
    	\end{enumerate}
\end{theorem}

Similarly, it can be proved that dependence logic $\fodl$ and exclusion logic $\fol(\mid)$ are complete with respect to all downwards closed existential second-order team properties, and that inclusion logic $\fol(\subseteq)$ is complete with respect to team properties definable in positive greatest fixed point logic -- we refer the reader to \cite{Vaananen2007-VNNDLA,kontinen2009definability,galliani2012inclusion, zbMATH06680142} for a proof of these results. In particular, this means that dependence, exclusion and inclusion logic are included in independence logic. Interestingly, the addition of either $\to$, $\cneg$ or $\unsim$ to $\indl$ makes the resulting system equi-expressive with full second-order logic \cite{zbMATH06175502, 660bf8a5bf224c62aeac78549c587a66}.  As in the present context we are interested in compactness, and full second-order logic is not compact, we will therefore exclude such systems from our considerations.

\subsection{Compactness in Team Semantics.}\label{Compactness in Team Semantics.}

The previous characterisation of the expressive power for logics of dependence motivates the interest on the issue of compactness. Like $\eso$, independence logic and its fragments are not abstract strong logics in Lindström's sense \parencite{lindstrom1969extensions}, and thus the fact that they extend $\fol$ is not an immediate reason to conclude that they are not compact, or do not satisfy the Löwenheim--Skolem theorem. In fact, it is easy to use the compactness theorem for first order logic to conclude that $\eso$, which is a proper extension of first-order logic, is compact. Similarly, there are several logics which are closely connected to logics of team semantics and which are known to be compact, e.g. first-order logic with Henkin quantifiers \cite{https://doi.org/10.1002/malq.19700160802, walkoe1970finite} and IF-logic \cite{hodges1997some}.
 
For logics over team semantics compactness turns out to be quite a subtle issue. Already in \cite{Vaananen2007-VNNDLA}, Väänänen proved that dependence logic is compact with respect to sentences, i.e. he proved the following claim.
\begin{theorem}[Väänänen]
    Let $\Gamma$ be a set of sentences of dependence logic. If every finite subset $\Gamma_0$ of $\Gamma$ is satisfiable, then $\Gamma$ is satisfiable. 
\end{theorem}
\noindent  Väänänen's proof is based on the translation to $\eso$, and on the fact that this latter system is compact. Compactness for sets of sentences of inclusion and independence logic follows in the same way by using the translation to existential second-order logic.
 
It should also be remarked that the previous version of compactness, which we can dub as “satisfiability compactness”, is not equivalent to the “consequence compactness”, namely the statement saying that, if $\Gamma\models \phi$, then there is some finite $\Gamma_0\subseteq \Gamma$ such that $\Gamma_0\models \phi$. This follows immediately from the fact~\parencite[§4.2]{Vaananen2007-VNNDLA} that there is a sentence $\Phi_\infty$ of dependence logic such that
\[
    \M\models\Phi_\infty \iff \text{$\M$ is infinite.}
\]

\noindent If we let $\phi_n$ be the first-order formula stating that there exists at least $n$-many elements, then $\{\phi_n \mid n<\omega  \}\models \phi_\infty$ but for any finite subset $\Gamma_0\subseteq\{\phi_n \mid n<\omega  \}$, we have that $\Gamma_0\nmodels \phi_\infty$.

Now, although compactness for sets of sentences of independence logic and its fragment has been known for long, it is not immediately obvious that the same property holds for sets of formulas with free variables as well. In fact, variables have a different role in team semantics compared to the standard Tarski semantics. In the context of first order logic, one can use the method of replacing free variables by fresh constant symbols to prove that compactness for formulas is equivalent to compactness for sentences.

In stronger logics with team semantics, however, the role of free variables cannot be mimicked by constants, as their interpretation may vary in different assignments of a team. Replacing a variable $x$ by a constant would be equivalent to adding the constancy atom $\dep(x)$ to the set of formulas one is inspecting. Next we present an example of a set of formulas where adding constancy atoms for the free variables destroys the satisfiability of the set.

\begin{example}
    Let $\Gamma = \{\forall y\ y\subseteq x, \exists y \exists z\ \neg y = z\}$. Clearly $\Gamma$ is satisfiable, as demonstrated by a model with domain $\{0,1\}$ and the team $\{\{(x,0)\}, \{(x,1)\}\}$. On the other hand, $\Gamma\cup\{\dep(x)\}$ is not satisfiable, which is seen as follows. Let $\M\models_X\Gamma$. Then $\M$ has at least two distinct elements, say $a$ and $b$, and $\M\models_{X(M/y)}y\subseteq x$. By the definition of duplication, both $a$ and $b$ occur as values of $y$ in $X(M/y)$, and by the semantics of the inclusion atom, they also must occur as values of $x$. But then there are $s,s'\in X$ such that $s(x) = a \neq b = s'(x)$. Thus $\M\nmodels_X\dep(x)$.
\end{example}

Interestingly, the use of the inclusion atom (or alternatively e.g. the independence atom) in the previous example turns out to be of fundamental importance. If we concentrate only on the  downwards closed fragment of independence logic, then it is possible to use constants to prove compactness for formulas in a method analogous to that of first-order logic. We first recall the following fact, which is an easy generalization of \cite[Lem. 3.28]{Vaananen2007-VNNDLA}.

\begin{fact}[Substitution Lemma]\label{substitution.lemma}
    Let $\phi\in \fol( \dep(\dots), \perp_c, \subseteq, |)$ and $t(x,\vec{y})$ a $\tau$-term. For any model $\M$ and team $X$,
    \[
    \M\models_X \phi(t/x) \Longleftrightarrow \M\models_{X(t/x)} \phi,
    \]
    where $X(t/x)=\{ s(s(t)/x)  \mid s\in X \}$ and $\phi(t/x)$ is the formula one obtains by replacing each free occurrence of $x$ in $\phi$ by $t$.
\end{fact}

\noindent If a logic is downwards closed, we can then derive the compactness theorem for formulas from that of sentences. 
\begin{theorem}
	The two following versions of compactness are equivalent for any downwards closed fragment $\mathcal{L}$ of $\fol( \dep(\dots), \perp_c, \subseteq, |)$.
	\begin{enumerate}
		\item Let $\Gamma$ be a set of formulas of $\mathcal{L}$. If every finite subset $\Gamma_0$ of $\Gamma$ is satisfiable, then $\Gamma$ is satisfiable. \label{DL-compactness for formulas}
		\item Let $\Gamma$ be a set of sentences of $\mathcal{L}$. If every finite subset $\Gamma_0$ of $\Gamma$ is satisfiable, then $\Gamma$ is satisfiable. \label{DL-compactness for sentences}
	\end{enumerate}
\end{theorem}
\begin{proof}
	Direction from~\ref{DL-compactness for formulas} to~\ref{DL-compactness for sentences} is obvious, for every sentence is a formula. Suppose for the converse that $\Gamma$ is a finitely consistent set of $\mathcal{L}$-formulas in some signature $\tau$. Let $x_i$, $i<\kappa$, enumerate all the free variables of $\Gamma$. We expand the signature $\tau$ with $\kappa$-many fresh constants $c_i$, $i<\kappa$, and we let $\Gamma'\coloneqq\{ \phi(\vec{c}/\vec{x}) \mid \phi \in \Gamma \}$ be the theory obtained by replacing every free variable $x_i$ by the corresponding constant $c_i$.
	
	Let $\Gamma'_0\subseteq\Gamma'$ be finite. Then we can find a finite $\Gamma_0\subseteq\Gamma$ such that $\Gamma'_0 = \{ \phi(\vec{c}/\vec{x}) \mid \phi \in \Gamma_0 \}$. By assumption, $\Gamma_0$ is satisfiable, and thus there is a $\tau$-structure $\M_0$ and a non-empty team $X_0$ such that $\M_0\models_{X_0}\Gamma_0$. Pick $s_0\in X_0$. Then, since $\mathcal{L}$ is downwards closed, it follows that $\M_0\models_{\{s_0\}}\Gamma_0$.
	
	Let $\N_0$ be the expansion of $\M_0$ to the signature $\tau\cup\{ c_i \mid i<\kappa \}$ such that for all $i<\kappa$, $c_i^{\N_0}=s_0(x_i)$. It clearly follows that $\N_0\models\Gamma'_0$, showing that $\Gamma'$ is finitely consistent.  By~\ref{DL-compactness for sentences} we then obtain that $\Gamma'$ is consistent, and hence for some $\N$ we have $\N\models\Gamma'$, i.e. $\N\models_{ \{\emptyset \}}\Gamma'$. Let $\M$ be the reduct of $\N$ to the signature $\tau$, and fix an assignment $s$ of $\M$ such that $s(x_i)=c_i^\N$ for all $i<\kappa$. By the Substitution Lemma we then obtain that $\M\models_{ \{s \} } \Gamma$, which proves our claim.
\end{proof}

In previous literature, there has been some interest in the issue of compactness for logics over team semantics. In particular, \textcite{luck2020team} has introduced a suitable notion of ultraproduct, which he used to prove the compactness of $\fol(\unsim)$, which he calls \emph{first-order team logic}. In a different direction, \textcite{https://doi.org/10.48550/arxiv.1904.08695} have provided a proof of compactness for sets of formulas of independence logic with countably many variables. In the following sections we build upon these previous accomplishments and provide two different proofs of the compactness theorem for formulas. In \cref{compactness-by-ultraproducts}, we build upon Martin Lück's ultraproduct construction, whilst in \cref{compactness-by-translation}, we use model-theoretic tools to generalize Kontinen and Yang's proof to arbitrary sets of formulas of independence logic.

\section{Compactness via Łoś' Theorem} \label{compactness-by-ultraproducts}

In this section, we recall the ultraproduct construction of \textcite{luck2020team} and generalize Łoś' Theorem for a variety of team-based extensions of first-order logic, providing a proof of the compactness theorem for each of them.

\subsection{Ultrafilters and Ultraproducts.} We briefly recall the notion of ultraproducts and introduce our notation regarding them. For a thorough introduction to ultraproducts in classical model theory, see e.g.~\textcite{Chang-Keisler}.

For any index set $I$, given a set $A_i$ for each $i\in I$, we denote by $\prod_{i\in I}A_i$ the Cartesian product of the sets $A_i$, i.e. the set of all functions $f$ with $\dom(f)=I$ such that $f(i)\in A_i$ for all $i\in I$. We also denote by $(a_i)_{i\in I}$ the element $f\in\prod_{i\in I}A_i$ such that $f(i)=a_i$ for all $i\in I$.

A \emph{(proper) filter} on $I$ is a family $\F\subseteq\Pow(I)$ such that
\begin{enumerate}
    \item $\emptyset\notin\F$ and $I\in\F$,
    \item for all $A,B\subseteq I$, if $A,B\in\F$, then $A\cap B\in\F$, and
    \item for all $A,B\subseteq I$, if $A\supseteq B\in\F$, then $A\in\F$.
\end{enumerate}
A filter $\U$ on $I$ is an \emph{ultrafilter} if for any $A\subseteq I$, either $A\in\U$ or $I\setminus A\in\U$. Note also that, as $A\cap (I\setminus A) = \emptyset$, it cannot be the case that they both are in $\U$. An ultrafilter is \emph{principal} if it contains a finite set. We are usually interested in non-principal ultrafilters.

A family $\F\subseteq\Pow(I)$ has the \emph{finite intersection property} if for all $n<\omega$ and $A_0,\dots,A_{n-1}\in\F$, we have $\bigcap_{i<n}A_i\neq\emptyset$. Clearly every filter has the finite intersection property, and it is easy to show that any family of subsets of $I$ with the finite intersection property can be extended into an ultrafilter on $I$.

Given $A_i$ for each $i\in I$ and an ultrafilter $\U$ on $I$, we define an equivalence relation
\[
	f\equiv g \mod\U \iff \{i\in I \mid f(i) = g(i)\}\in\U
\]
on $\prod_{i\in I}A_i$. We denote by $\prod_{i\in I}A_i/\U$ its set of equivalence classes and by $f/\U$ the equivalence class of $f\in\prod_{i\in I}A_i$.  

\begin{definition}[Ultraproduct of structures]
	Let $\tau$ be a vocabulary and $\M_i$ a $\tau$-structure for each $i\in I$. We define the $\tau$-structure $\M\coloneqq\prod_{i\in I}\M_i/\U$ as follows and call it the \emph{ultraproduct} of $\M_i$.
	\begin{itemize}
		\item The domain of $\M$ is the set $M\coloneqq\prod_{i\in I}M_i/\U$, where $M_i = \dom(\M_i)$.
		\item Given a constant symbol $c\in\tau$, we let $c^\M = (c^{\M_i})_{i\in I}/\U$.
		\item Given an $n$-ary relation symbol $R\in\tau$, we let $R^\M$ be the set of all tuples $(f_0/\U,\dots,f_{n-1}/\U)$ such that
		\[
			\{i\in I \mid (f_0(i),\dots,f_{n-1}(i))\in R^{\M_i}\}\in\U.
		\]
		\item Given an $n$-ary function symbol $F\in\tau$, we let
		\[
    		F^\M(f_0/\U,\dots,f_{n-1}/\U) = (F^{\M_i}(f_0(i),\dots,f_{n-1}(i)))_{i\in I}/\U.
		\]
	\end{itemize}
	If each $\M_i$ is the same model $\M_0$, we call $\M$ the \emph{ultrapower} of $\M_0$ and denote it by $\M_0^I/\U$.
\end{definition}
\medskip

 \emph{For the rest of this section, we fix a vocabulary $\tau$, an index set $I$, $\tau$-structures $\M_i$, $i\in I$, and an ultrafilter $\U$ on $I$, and we denote by $\M$ the ultraproduct $\prod_{i\in I}\M_i/\U$.} 
 
 \medskip
 
 Given assignments $s_i\colon D\to\M_i$, $i\in I$, we denote by $(s_i)_{i\in I}$ the assignment $s\colon D\to\prod_{i\in I}\M_i$ such that $s(x) = (s_i(x))_{i\in I}$ for all $x\in D$. Given an assignment $s\colon D\to\prod_{i\in I}\M_i$, we denote by $s/\U$ the assignment $t\colon D\to\M$ such that $t(x)=s(x)/\U$ for all $x\in D$. Next we define the ultraproduct of teams, first introduced in~\cite[Def. 5.32]{luck2020team}.
\begin{definition}[Ultraproduct of teams]
 	Let $D$ be a (possibly infinite) set of variables. Given a team $X_i$ of $\M_i$ with domain $D$ for each $i\in I$, we define its \emph{team ultraproduct} $\prod_{i\in I}X_i/\U$ as  the set of all assignments $s\colon D\to\M$ such that there are $s_i\colon D\to\M_i$, $i\in I$, with $s = (s_i)_{i\in I}/\U$ and $\{i\in I \mid s_i\in X_i\}\in\U$.
\end{definition}

\begin{lemma}\label{Basic lemma about assignments}\quad
	\begin{enumerate}
		\item Every assignment $s\in\left(\prod_{i\in I}M_i\right)^D$ is of the form $(s_i)_{i\in I}$ for some $s_i\in M_i^D$, and every assignment $s\in M^D$ is of the form $t/\U$ for some $t\in\left(\prod_{i\in I}M_i\right)^D$. \label{Assignments are of the expected form}
		
		\item Given $s\in\left(\prod_{i\in I}M_i\right)^D$, $f\in\prod_{i\in I}M_i$ and $x\in D$, we have
		\[
			\frac{s}{\U}\!\left( \frac{f}{\U} \bigg / x \right) = s(f/x)/\U.
		\]
		\label{Modified assignment as the equivalence class of modified assignments}
		
		\item Let $s\in\left(\prod_{i\in I}M_i\right)^D$, $f\in\prod_{i\in I}M_i$ and $x\in D$, denote $t = s(f/x)$, and let $s_i,t_i\in M_i^D$ be such that $s=(s_i)_{i\in I}$ and $t=(t_i)_{i\in I}$. Then $t_i = s_i(f(i)/x)$. \label{Modified assignment as the sequence of modified quotient assignments}
		
		\item Let $s=(s_i)_{i\in I}/\U$ and $t=(t_i)_{i\in I}/\U$ be assignments of $\M$ with domain $D$. Then $s=t$ if $\{i\in I \mid s_i = t_i\}\in\U$. \label{Sufficient condition for two ultraproduct assignments being the same}
	\end{enumerate}
\end{lemma}
\begin{proof}
    Routine work with ultraproducts.
\end{proof}

The following basic lemma will be crucial later.
\begin{lemma}\label{Basic lemma about teams}
	Let $X_i$ and $Y_i$ be teams over $\M_i$, let $a=(a_i)_{i\in I}/\U \in\M$ and let $F_i\colon X_i\to\wp^+(M_i)$ be a supplement function of $X_i$. Denote $X=\prod_{i\in I}X_i/\U$ and $Y=\prod_{i\in I}Y_i/\U$. Then the following hold.
	\begin{enumerate}
		\item $X\cup Y = \prod_{i\in I}(X_i\cup Y_i)/\U$; moreover, if $X_i$ and $Y_i$ are disjoint for $\U$-many $i$, then $X$ and $Y$ are disjoint, \label{Tensoring in quotient teams}
		\item $X(a/x) = \prod_{i\in I}X_i(a_i/x)/\U$, \label{Simple supplementation in quotient teams}
		\item $X(M/x) = \prod_{i\in I}X_i(M_i/x)/\U$, and \label{Duplication in quotient teams}
		\item $X(F/x) = \prod_{i\in I}X_i(F_i/x)/\U$, where $F\colon X\to\wp^+(\M)$ is such that
		\[
			F((s_i)_{i\in I}/\U) = \{f/\U \mid f\in\prod_{i\in I}F_i(s_i) \};
		\]
		moreover, if for $\U$-many $i$, $F_i(s)$ is a singleton for all $s\in X_i$, then $F(s)$ is a singleton for all $s\in X$.
		\label{Supplementation in quotient teams}
	\end{enumerate}
\end{lemma}
\begin{proof}\quad
	\begin{enumerate}
		\item Let $s = (s_i)_{i\in I}/\U \in X$. As $X_i\subseteq X_i\cup Y_i$ for all $i\in I$, we have $s_i\in X_i\cup Y_i$, whence $s\in \prod_{i\in I}(X_i\cup Y_i)/\U$. Thus $X\subseteq\prod_{i\in I}(X_i\cup Y_i)/\U$. Similarly $Y\subseteq\prod_{i\in I}(X_i\cup Y_i)/\U$, whence $X\cup Y\subseteq\prod_{i\in I}(X_i\cup Y_i)/\U$.
		
		For the converse, suppose that $s = (s_i)_{i\in I}/\U\in\prod_{i\in I}(X_i\cup Y_i)/\U$. Now there is $J_0\in\U$ such that $s_i\in X_i\cup Y_i$ for all $i\in J_0$. We wish to show that $s\in X\cup Y$. For this, it is enough to show that either of the sets
		\[
			\{i\in I \mid s_i \in X_i \} \quad\text{and}\quad \{i\in I \mid s_i \in Y_i \}
		\]
		is in $\U$. So suppose not: then both of the sets
		\[
			\{i\in I \mid s_i \notin X_i \} \quad\text{and}\quad \{i\in I \mid s_i \notin Y_i \}
		\]
		and thus their intersection 
		$J_1 \coloneqq \{i\in I \mid s_i \notin X_i\cup Y_i\}$ are in $\U$. But then we have $s_i\in X_i\cup Y_i$ and $s_i\notin X_i\cup Y_i$ for all $i\in J_0\cap J_1$, whence $J_0\cap J_1 = \emptyset$, which is a contradiction because $\emptyset\notin\U$. Thus one of the claimed sets is in the ultrafilter, which proves our claim.
		
		For the ``moreover'' part, suppose that $J_0\coloneqq\{i\in I \mid \text{$X_i$ and $Y_i$ are disjoint} \}\in\U$. Suppose for a contradiction that there is $s=(s_i)_{i\in I}/\U\in X\cap Y$. A dual argument to the above shows that $X\cap Y = \prod_{i\in I}(X_i\cap Y_i)/\U$, so now there is $J_1\in\U$ such that $s_i\in X_i\cap Y_i$ for all $i\in J_1$. But then again, $\emptyset = J_0\cap J_1\in\U$, a contradiction.
		
		\item Follows from~\ref{Supplementation in quotient teams}, as $X(a/x)$ is just $X$ supplemented by the constant function $F(s) = \{a\}$.
		
		\item Follows from~\ref{Supplementation in quotient teams}, as duplication is supplementation by the constant function $F(s) = M$.
		
		\item Let $s/\U\in X(F/x)$. Then there is $t/\U\in X$ and $f/\U\in F(t/\U)$ such that
		\[
			s/\U = \frac{t}{\U}\!\left( \frac{f}{\U}\bigg/ x \right).
		\]
		Let $t_i\in \M_i^D$ be such that $t=(t_i)_{i\in I}$. Now by the definition of ultraproduct of teams and Lemma~\ref{Basic lemma about assignments}~\ref{Sufficient condition for two ultraproduct assignments being the same}, there is $J_0\in\U$ such that $t_i\in X_i$ for all $i\in J_0$. By the definition of $F$, there is $J_1\in\U$ such that $f(i)\in F_i(t_i)$ for all $i\in J_1$. Thus for $i\in J_0\cap J_1$ we have $t_i(f(i)/x)\in X_i(F_i/x)$ and hence, by definition of ultraproduct of teams,
		\[
			(t_i(f(i)/x))_{i\in I}/\U\in\prod_{i\in I}X_i(F_i/x)/\U.
		\]
		By Lemma~\ref{Basic lemma about assignments}~\ref{Modified assignment as the equivalence class of modified assignments} and~\ref{Modified assignment as the sequence of modified quotient assignments}, $s/\U = t(f/x)/\U = (t_i(f(i)/x))_{i\in I}/\U$, whence we conclude $s/\U \in \prod_{i\in I}X_i(F_i/x)/\U$ as desired.
		
		For the converse, let $s/\U\in\prod_{i\in I}X_i(F_i/x)/\U$. Then $s=(s_i)_{i\in I}$ for some $s_i$ such that there is $J\in\U$ with $s_i\in X_i(F_i/x)$ for all $i\in J$. Now, for each $i\in J$ there is $t_i\in X_i$ and $a_i\in F_i(t_i)$ with $s_i = t_i(a_i/x)$. Define $f\in\prod_{i\in I}\M_i$ by setting $f(i)=a_i$ for $i\in J$ and $f(i) = s_i(x)$ for $i\in I\setminus J$. Let $t=(t_i)_{i\in I}$, where $t_i = s_i\restriction\dom(X)$ for $i\in I\setminus J$. Then by Lemma~\ref{Basic lemma about assignments}~\ref{Modified assignment as the sequence of modified quotient assignments}, $t(f/x) = (t_i(f(i)/x))_{i\in I} = (s_i)_{i\in I} = s$. By Lemma~\ref{Basic lemma about assignments}~\ref{Modified assignment as the equivalence class of modified assignments}, $s/\U = t(f/x)/\U = \frac{t}{\U}(\frac{f}{\U}/x)$, and as $t_i\in X_i$ for all $i\in J$, $t/\U\in X$. As $f(i)=a_i\in F_i(t_i)$ for each $i\in J$, we have $f/\U\in F(t/\U)$, whence $s/\U\in X(F/x)$ as desired.
		
		For the ``moreover'' part, suppose that
		\[
		    J_0\coloneqq\{i\in I \mid \text{$F_i(s)$ is a singleton for all $s\in X_i$} \}\in\U.
		\]
		Now for all $i\in J_0$ and $s\in X_i$, there are $a_{i,s}$ such that $F_i(s)=\{a_{i,s}\}$. Let $s\in X$. Now there are $s_i$ and $J_1\in\U$ such that $s=(s_i)_{i\in I}/\U$ and $s_i\in X_i$ for all $i\in J_1$. Now $J\coloneqq J_0\cap J_1\in\U$, and for all $i\in J$, $F_i(s_i) = \{a_{i,s_i}\}$. For $i\in I\setminus J_0$, let $a_{i,s_i}\in\M$ be arbitrary, and let $a = (a_{i,s_i})_{i\in I}$. Now for every $b=(b_i)_{i\in I}\in\prod_{i\in I}F_i(s_i)$, we have $b_i = a_i$ for all $i\in J$ and hence $b/\U = a/\U$. Thus $F(s) = \{a/\U\}$. \qedhere
	\end{enumerate}
\end{proof}

\subsection{Preservation in Ultraproducts.}	We define what it means for a team property to be (strongly) closed under ultraproducts, and what it means for an operation on team properties to preserve (strong) closure under ultraproducts. The notion of being strongly closed under ultraproducts was originally introduced by Lück in~\cite{luck2020team}.\footnote{What we have decided to call ``being strongly closed under ultraproducts'' here, is called ``being preserved in ultraproducts'' in~\cite{luck2020team}.}

\begin{definition}\quad
    \begin{enumerate}
        \item Let $\P$ be a team property with domain $D$. We say that $\P$ is \emph{closed under ultraproducts} if for all sets $I$, structures $\M_i$ and teams $X_i$ of $\M_i$ with domain $D$, $i\in I$, and ultrafilters $\U\subseteq\Pow(I)$, we have
        \[
            \{i\in I \mid (\M_i, X_i) \in \P \}\in\U \implies \left(\prod_{i\in I}\M_i/\U, \prod_{i\in I}X_i/\U \right)\in \P.
        \]

        \item We say that a team property $\P$ is \emph{strongly closed under ultraproducts} if both $\P$ and its complement $\P^c$ are closed under ultraproducts.

        \item We say that an $n$-ary operation $f$ on team properties \emph{preserves being closed under ultraproducts} if whenever $\P_i$, $i<n$, are closed under ultraproducts, also $f(\P_0,\dots,\P_{n-1})$ is. We say that $f$ \emph{preserves being strongly closed under ultraproducts} if whenever $\P_i$, $i<n$, are strongly closed under ultraproducts, also $f(\P_0,\dots,\P_{n-1})$ is.
    \end{enumerate}
\end{definition}

\noindent The following theorem was proved in \textcite[Lem. 5.33, Thm 5.36]{luck2020team}.
\begin{theorem}[Lück]\label{Atoms are strongly preserved}\quad
    \begin{enumerate}
        \item Flat team properties are strongly closed under ultraproducts.
        \item First-order team properties are strongly closed under ultraproducts. In particular, dependence, independence, inclusion and exclusion atoms are strongly closed under ultraproducts.
    \end{enumerate}
\end{theorem}
\noindent  Considering the weaker notion that we call being closed under ultraproducts, we can prove a version of Łoś' Theorem for a larger class of logics. First, we consider those operations which preserve both being closed under ultraproducts and also being strongly closed under ultraproducts.

\begin{lemma}\label{Operations strongly preserved}
    The following operations on team properties preserve both being closed and being strongly closed under ultraproducts:
    \begin{enumerate}
        \item conjunction, i.e. the binary operation
        \[
            \P\land\Q = \P\cap\Q,
        \]
        \item intuitionistic disjunction, i.e. the binary operation
        \[
            \P\ivee\Q = \P\cup\Q,
        \]
        \item universal quantifiers, i.e. the unary operations
        \[
            \forall x(\P) = \{ (\M, X) \mid (\M,X(M/x))\in\P \}
        \]
        for every $x\in D$, where $D$ is the domain of $\P$,
        \item universal $1$-quantifiers, i.e. the unary operations
        \[
            \forallone x(\P) = \{ (\M, X) \mid \text{$(\M,X(a/x))\in\P$ for all $a\in M$} \}
        \]
        for every $x\in D$, where $D$ is the domain of $\P$, and
        \item existential $1$-quantifiers, i.e. the unary operations
        \[
            \existsone x(\P) = \{ (\M, X) \mid \text{$(\M,X(a/x))\in\P$ for some $a\in M$} \}
        \]
        for every $x\in D$, where $D$ is the domain of $\P$.
    \end{enumerate}
\end{lemma}
\begin{proof}
    We prove the cases for conjunction and universal quantifier; the other ones are similar. We show that the operations preserve being strongly closed in ultraproducts. In each case, it is easy to see how to modify the proof so that it shows that the operation in question preserves being closed under ultraproducts.

    For the proof, let $I$ be an index set, $\M_i$ a structure and $X_i$ a team of $\M_i$ with domain $D$ for $i\in I$, and $\U$ an ultrafilter on $I$. Denote $\M\coloneqq\prod_{i\in I}\M_i/\U$ and $X\coloneqq\prod_{i\in I}X_i/\U$.

    \begin{enumerate}
        \item Suppose that $\P$ and $\Q$ are strongly closed under ultraproducts. From the properties of ultrafilters, it follows that
		\begin{align*}
		    \{i\in &I \mid (\M_i, X_i)\in\P\land\Q \}\in\U \\
		    &\iff \{i\in I \mid (\M_i, X_i)\in\P\}\cap\{i\in I \mid (\M_i, X_i)\in\Q\}\in\U \\
		    &\iff \{i\in I \mid (\M_i, X_i)\in\P\}\in\U\ \text{and}\ \{i\in I \mid (\M_i, X_i)\in\Q\}\in\U \\
		    &\iff (\M,X)\in\P\ \text{and}\ (\M,X)\in\Q \\
		    &\iff (\M,X)\in\P\land\Q.
		\end{align*}
		Hence $\P\land\Q$ is strongly closed under ultraproducts.

		\setcounter{enumi}{2}

        \item Suppose that $\P$ is strongly closed under ultraproducts. First suppose that
        \[
            \{i\in I \mid (\M_i,X_i)\in \forall x(\P)\}\in\U.
        \]
        Then $(\M_i,X_i(M_i/x))\in\P$ for all $i\in\{i\in I \mid (\M_i,X_i)\in \forall x(\P)\}$, and hence
        \[
            \{i\in I \mid (\M_i,X_i(M_i/x))\in\P \}\supseteq\{i\in I \mid (\M_i,X_i)\in \forall x(\P)\}.
        \]
        By upwards closedness of $\U$, now
        \[
            \{i\in I \mid (\M_i,X_i(M/x))\in\P \}\in\U.
        \]
        As $\P$ is closed under ultraproducts, we have $(\M,\prod_{i}X_i(M_i/x)/\U)\in\P$. By Lemma~\ref{Basic lemma about teams}~\ref{Duplication in quotient teams},
        \[
            X(M/x) = \prod_{i\in I}X_i(M_i/x)/\U,
        \]
        and thus $(\M,X(M/x))\in\P$. But this means that $(\M,X)\in\forall x(\P)$. Hence $\forall x(\P)$ is closed under ultraproducts.
		
		Then suppose that
        \[
            \{i\in I \mid (\M_i,X_i)\in (\forall x(\P))^c \}\in\U.
        \]
		Now for all $i\in\{i\in I \mid (\M_i,X_i)\notin \forall x(\P)\}$, we have $(\M_i,X_i)\notin \forall x(\P)$, i.e. $(\M_i,X_i(M_i/x))\notin\P$, whence
		\[
		    \{i\in I \mid (\M_i,X_i(M_i/x))\in\P^c \}\in\U.
		\]
		As $\P^c$ is closed under ultraproducts, we have $(\M,\prod_{i}X_i(M/x)/\U)\in\P^c$, whence by Lemma~\ref{Basic lemma about teams}~\ref{Duplication in quotient teams} we then obtain that $(\M,X(M/x))\in\P^c$. Thus $(\M,X)\in(\forall x(\P))^c$. This shows that $(\forall x(\P))^c$ is closed under ultraproducts, and thus $\forall x(\P)$ is strongly closed under ultraproducts. \qedhere

    \end{enumerate}
\end{proof}

Next, the following lemma proves that (weak) classical negation preserves being strongly closed under ultraproduct.

\begin{lemma}\label{operations.only.strong.preservation}
    The following operations on team properties preserve being strongly closed under ultraproducts:
    \begin{enumerate}
        \item the weak classical negation, i.e. the unary operation
        \[
            \cneg\P \coloneqq \{(\M,X) \mid \text{$X=\emptyset$ or $(\M,X)\notin\P$}\},
        \]
        and
        \item the classical negation, i.e. the unary operation
        \[
            \unsim\P \coloneqq \P^c.
        \]
    \end{enumerate}
\end{lemma}
\begin{proof}
    For the proof of preserving being strongly closed under ultraproducts, let $I$ be an index set, $\M_i$ a structure and $X_i$ a team of $\M_i$ with domain $D$ for $i\in I$, and $\U$ an ultrafilter on $I$. Denote $\M\coloneqq\prod_{i\in I}\M_i/\U$ and $X\coloneqq\prod_{i\in I}X_i/\U$.
    \begin{enumerate}
        \item Suppose that $\P$ is strongly closed under ultraproducts and that
        \[
            J_0 \coloneqq \{i\in I \mid (\M_i,X_i)\in\cneg\P\}\in\U.
        \]
        Now if $X = \emptyset$, we have $(\M,X)\in\cneg\P$ and we are done, so suppose that $X\neq\emptyset$. Then there is $J_1\in\U$ such that for all $i\in J_1$, $X_i\neq\emptyset$. Now $J\coloneqq J_0\cap J_1\in\U$, and for all $i\in J$, $(\M_i,X_i)\in\P^c$. As $\P^c$ is closed under ultraproducts, it follows that $(\M,X)\in\P^c$, whence $(\M,X)\in\cneg\P$. Thus $\cneg\P$ is closed under ultraproducts.
        
        Then suppose that
        \[
            J \coloneqq \{i\in I \mid (\M_i,X_i)\in(\cneg\P)^c\}\in\U.
        \]
        Now for all $i\in J$, $\emptyset\neq X_i\in\P$. But then $X\neq\emptyset$ and as $\P$ is closed under ultraproducts, $(\M,X)\in\P$. Hence $(\M,X)\in(\cneg\P)^c$, and $(\cneg\P)^c$ is closed under ultraproducts. This shows that $\cneg\P$ is strongly closed under ultraproducts.

        \item If $\P$ is strongly closed under ultraproducts, then, just by definition, $\unsim \P$ also is.  \qedhere
    \end{enumerate}
\end{proof}

\noindent To see that neither $\cneg$ nor $\unsim$ preserves being closed under ultraproducts, consider the following example.

\begin{example}\label{cneg.example}
     Fix the signature $\tau=\{ < \}$ and define $\P_\omega$ with domain $\{x\}$ as follows. Denote by $T$ the first-order $\tau$-theory of $\mathbb{N}$. For any pair $(\N, X)$, we let $(\N, X)\in \P_\omega$ if either $\N\nmodels T$, or $\N\models T$ and every $a\in X[x]$ is a non-standard number, meaning that there are infinitely many elements strictly less than $a$. Now, notice that the property of being non-standard can be expressed in first-order logic by an infinite set of formulas $\psi_n(x)$ for $n<\omega$, where
    \[
     \psi_n = \exists x_0 \dots \exists x_{n-1} \left(\bigwedge_{i<j<n} \neg x_i=x_j \land  \bigwedge_{i<n} x_i < x \right).
     \]
     It thus follows immediately by Łoś' Theorem for first-order logic that $\P_\omega$ is closed under ultraproducts. However, consider now the team property $\unsim\P_\omega$. Let $\mathbb{N}^\omega/\U$ be an ultrapower of the natural numbers by a non-principal ultrafilter $\U$, and, for each $n<\omega$, let $s_n\colon\{x\}\to\mathbb{N}$ be the assignment $x\mapsto n$. Then, we have that $\{n<\omega \mid (\mathbb{N}, \{s_n\})\in\unsim\P_\omega \} = \omega \in\U$. However, we also obtain $(\mathbb{N}^\omega/\U,\prod_{n<\omega} \{s_n\}/\U)\notin \unsim\P_\omega$, as the element $(n)_{n<\omega}/\U$ is non-standard. The case of $\cneg\P$ is handled analogously.
\end{example}

Finally, we consider those operations which only preserve being closed under ultraproducts.

\begin{lemma}\label{Operations preserved}
    The following operations on team properties preserve being closed under ultraproducts:
    \begin{enumerate}
        \item tensor disjunction with lax semantics, i.e. the binary operation
        \[
            \P\vee\Q = \{ (\M, X\cup Y) \mid (\M,X)\in\P, (\M,Y)\in\Q \},
        \]
        \item tensor disjunction with strict semantics, i.e. the binary operation
        \[
            \P\svee\Q = \{ (\M, X\cup Y) \mid (\M,X)\in\P, (\M,Y)\in\Q, X\cap Y=\emptyset \},
        \]
        \item existential quantifiers with lax semantics, i.e. the unary operations
        \[
            \exists x(\P) = \{ (\M, X) \mid \text{$(\M,X(F/x))\in\P$ for some $F\colon X\to\Pow^+(M)$} \}
        \]
        for every $x\in D$, where $D$ is the domain of $\P$, and
        \item existential quantifiers with strict semantics, i.e. the unary operations
        \[
            \sexists x(\P) = \{ (\M, X) \mid \text{$(\M,X(F/x))\in\P$ for some $F\colon X\to M$} \}
        \]
        for every $x\in D$, where $D$ is the domain of $\P$.
    \end{enumerate}
\end{lemma}
\begin{proof}
    For the proof, let $I$ be an index set, $\M_i$ a structure and $X_i$ a team of $\M_i$ with domain $D$ for $i\in I$, and $\U$ an ultrafilter on $I$. Denote $\M\coloneqq\prod_{i\in I}\M_i/\U$ and $X\coloneqq\prod_{i\in I}X_i/\U$.
    \begin{enumerate}
        \item Suppose that $\P$ and $\Q$ are closed under ultraproducts, and that
        \[
            J\coloneqq\{i\in I \mid (\M_i,X_i)\in\P\vee\Q\}\in\U.
        \]
        In order to show that $(\M,X)\in\P\vee\Q$, we need to find teams $Y$ and $Z$ of $\M$ such that $Y\cup Z = X$ and $(\M,Y)\in\P$ and $(\M,Z)\in\Q$. Now, for all $i\in J$, the team $X_i$ can be divided into two teams $Y_i$ and $Z_i$ of $\M_i$ such that $(\M_i,Y_i)\in\P$ and $(\M_i,Z_i)\in\Q$. For $i\in I\setminus J$, let $Y_i = Z_i = X_i$ and let $Y=\prod_{i\in I}Y_i/\U$ and $Z=\prod_{i\in I}Z_i/\U$. We show that these $Y$ and $Z$ suffice.
		
		Clearly $Y$ and $Z$ are teams of $\M$ with domain $D$, and by Lemma~\ref{Basic lemma about teams}~\ref{Tensoring in quotient teams} $X=Y\cup Z$. Also, since for any $i\in J$ we have $(\M_i,Y_i)\in\P$ and $(\M_i,Z_i)\in\Q$, we get
		\[
		    \{i\in I \mid (\M_i,Y_i)\in\P\}\supseteq J\ \text{and}\ \{i\in I \mid (\M_i,Z_i)\in\Q\}\supseteq J,
		\]
		whence
		\[
		    \{i\in I \mid (\M_i,Y_i)\in\P\}\in\U\ \text{and}\ \{i\in I \mid (\M_i,Z_i)\in\Q\}\in\U.
		\]
		Then, as $\P$ and $\Q$ are closed under ultraproducts, we have $(\M,Y)\in\P$ and $(\M,Z)\in\Q$, as desired.
		
		\item In the proof for the lax tensor, insert the assumption of disjointness of $Y_i$ and $Z_i$, and Lemma~\ref{Basic lemma about teams}~\ref{Tensoring in quotient teams} ensures that also $Y$ and $Z$ are disjoint. This yields a proof for the strict tensor.
		
		\item Suppose that $\P$ is closed under ultraproducts, and that
        \[
            J\coloneqq\{i\in I \mid (\M_i,X_i)\in \exists x(\P)\}\in\U.
        \]
		Now for all $i\in J$, we have $(\M_i,X_i(F_i/x))$ for some function $F_i\colon X_i\to \Pow^+(\M_i)$. For $i\in I\setminus J$, let $F_i(s)=M$ for all $s\in X_i$. We then have that
		\[
		    \{i\in I \mid (\M_i,X_i(F_i/x))\in\P\}\supseteq J
		\]
		and hence
		\[
		    \{i\in I \mid (\M_i,X_i(F_i/x))\in\P\}\in\U.
		\]
		As $\P$ is closed under ultraproducts, we have  $(\M,\prod_{i}X_i(F_i/x))\in\P$. By Lemma~\ref{Basic lemma about teams}~\ref{Supplementation in quotient teams} we then obtain that $(\M,X(F/x))\in\P$, where $F$ is defined as in the statement of the lemma. It follows that $(\M,X)\in\exists x(\P)$, which proves our claim.
		
		\item In the proof for lax existential quantifier, insert the assumption that each $F_i$ maps every assignment to a singleton, and Lemma~\ref{Basic lemma about teams}~\ref{Supplementation in quotient teams} ensures that also $F$ maps every assignment to a singleton. This yields a proof for the strict existential quantifier. \qedhere
    \end{enumerate}
\end{proof}

\subsection{Łoś' Theorem and Compactness.} We are now in a place to prove a version of Łoś' Theorem for several extensions of first-order logic via team semantics, and thus derive compactness. We first define what it means for a logic to have a Łoś' Theorem.

\begin{definition}
    We say that a logic $\mathcal{L}$ has a \emph{Łoś' Theorem} if the following holds.
    \begin{quote}
        Let $I$ be a set, for each $i\in I$ let $\M_i$ be a $\tau$-structure and $X_i$ a team of $\M_i$ with a shared domain $D$, and let $\U$ be an ultrafilter on $I$. Denote $\M \coloneqq \prod_{i\in I}\M_i/\U$ and $X \coloneqq \prod_{i\in I}X_i/\U$. Then, given any $\tau$-formula $\phi$ of the logic $\mathcal{L}$ such that the free variables of $\phi$ are contained in $D$, we have
		\[
    		\{i\in I \mid \M_i\models_{X_i}\phi\}\in\U \implies \M\models_X \phi.
		\]
    \end{quote}
    We say that $\mathcal{L}$ has a \emph{strong Łoś' Theorem} if the above holds in both directions:
		\[
    		\{i\in I \mid \M_i\models_{X_i}\phi\}\in\U \iff \M\models_X \phi.
		\]
\end{definition}

\noindent We can then derive from our previous results the following propositions, which correspond to a version of Łoś' Theorem in our context.

\begin{proposition}\label{Abstract Los}
    Let $\mathcal{L}$ be a team-semantic extension of first-order logic whose atomic formulas are (strongly) closed under ultraproducts and whose logical operations preserve being (strongly) closed under ultraproducts. Then $\mathcal{L}$ has a (strong) Łoś' Theorem.
\end{proposition}
\begin{proof}
    Trivial induction.
\end{proof}

\begin{corollary}\label{team.logic.corollary} $\;$
    \begin{enumerate}
        \item Let $C\subseteq\{\dep(\dots), \perp_c, \subseteq, |\}$. Then $\fol(C)$, both with strict and lax semantics, has a Łoś' Theorem.
        \item The fragment of $\foil$ that contains neither the tensor disjunction $\lor$ nor the existential quantifier $\exists$ has a strong Łoś' Theorem.
    \end{enumerate}
\end{corollary}
\begin{proof}
    After~\cref{Abstract Los}, follows from \cref{Atoms are strongly preserved,Operations strongly preserved,operations.only.strong.preservation,Operations preserved}.
\end{proof}

Finally, we obtain a compactness theorem for formulas of a whole family of extensions of first-order logic under team semantics. The following proof of compactness from Łoś' Theorem is standard and can be found e.g. in \cite{Chang-Keisler}.

\begin{theorem}[Compactness]
	Let $\mathcal{L}$ be a logic with a Łoś' Theorem, and let $\Gamma$ be a set of formulas of $\mathcal{L}$. If every finite subset $\Gamma_0$ of $\Gamma$ is satisfiable, then $\Gamma$ is satisfiable. In particular, $\fol(C)$ is compact,  both with strict and lax semantics, whenever $C\subseteq\{\dep(\dots), \perp_c, \subseteq, |\}$.
\end{theorem}
\begin{proof}
    We prove only the first claim as the second claim follows immediately from the first and \cref{team.logic.corollary}.  Let $\mathcal{S}$ be the set of finite subsets of $\Gamma$. By the assumption, for any $S\in\mathcal{S}$ there is a model $\M_S$ and a nonempty team $X_S$ such that $\M_S\models_{X_S} S$. For $\phi\in\Gamma$, define $[\phi] \coloneqq \{S\in\mathcal{S} \mid \phi\in S\}$ and let $F = \{ [\phi] \mid \phi\in\Gamma\}$. Now $F$ has the finite intersection property: if $\phi_0,\dots,\phi_{n-1}\in\Gamma$, then
    \[
        \bigcap_{i<n}[\phi_i] = \{S\in\mathcal{S} \mid \phi_0,\dots,\phi_{n-1}\in S\}\ni\{\phi_0,\dots,\phi_{n-1}\},
    \]
    and hence $\bigcap_{i<n}[\phi_i]\neq\emptyset$. Let $\U$ be an ultrafilter on $\mathcal{S}$ extending $F$. Let
    \[
        \M = \prod_{S\in\mathcal{S}}\M_S/\U \quad\text{and}\quad X = \prod_{S\in\mathcal{S}}X_S/\U.
    \]
    Now for any $\phi\in\Gamma$, we have $\{S\in\mathcal{S} \mid \M_S\models_{X_S}\phi\}\supseteq [\phi]\in\U$, whence $\{S\in\mathcal{S} \mid \M_S\models_{X_S}\phi\}\in\U$. Then by Łoś' Theorem, $\M\models_X\phi$ for all $\phi\in\Gamma$, i.e. $\M\models_X\Gamma$. As $X$ is nonempty, this shows that $\Gamma$ is satisfiable.
\end{proof}

\section{Compactness via $\eso$-Translation}\label{compactness-by-translation}

In this section, we provide a proof of compactness for several logics over team semantics that generalizes the proof of \textcite{https://doi.org/10.48550/arxiv.1904.08695} for compactness for sets of formulas with only countably many free  variables. Even though we also have a proof for compactness using ultraproducts, we believe the proof ideas presented here may be of separate interest. We are grateful to Tapani Hyttinen for suggesting this approach for the proof. 

We shall assume in this section that $\mathcal{L}$ is always a logic over team semantics which is local and admits a translation to $\eso$ in the sense of \cref{translation} -- it is clear that independence logic and its fragments are examples of such logics. 

First, we fix some preliminary notation and terminology.
\begin{definition}\label{notation}
    Suppose $\Gamma$ is a set of $\tau$-formulas of  $\mathcal{L}$ and let $x_i$, $i<\kappa$, enumerate the free variables occurring in formulas of $\Gamma$.
    \begin{enumerate}
        \item Given $\phi\in\Gamma$, we denote by $I_\phi$ the set of all indices $i\in\kappa$ such that $x_i$ occurs free in $\phi$.

        \item Given $I\subseteq\kappa$, let $\Gamma_I=\{ \phi \in \Gamma \mid I_\phi\subseteq I \}$ be the set of formulas whose free variables are indexed by elements of $I$.

        \item For each $\phi\in\Gamma$, we let $R_\phi$ be a fresh $|I_\phi|$-ary relation symbol. We denote by $\chi_\phi(R_\phi)$ a translation of $\phi$ to $\eso$ of the form $\exists R^1_\phi \dots \exists R^{n-1}_\phi \alpha_\phi$ where $\alpha_\phi$ is a first-order sentence in the vocabulary $\tau\cup\{R_\phi, R^0_\phi,\dots,R^{n-1}_\phi\}$ and both $R_\phi$ and each $R^i_\phi$ are fresh.

        \item For every finite $I\subseteq \kappa$, we let $S_I$ be a fresh $|I|$-ary predicate symbol and
        \[
            \tau_\Gamma = \tau\cup\{R_\phi \mid \phi\in \Gamma\}\cup \{S_I \mid I\subseteq \kappa, |I|<\omega \}.
        \]
        We write $S_I((x_i)_{i\in I})$ for the formula $S_I(x_{i_0},\dots,x_{i_{n-1}})$, where $i_0,\dots,i_{n-1}$ is an enumeration of $I$ in increasing order.
        
        \item We let $\Delta_\Gamma$ be the following set of $\tau_\Gamma$-sentences:
        \begin{itemize}
            \item $\exists \vec{v} R_\phi(\vec{v})$,
            \item $\forall(x_i)_{i\in I_\phi} (R_\phi((x_i)_{i\in I\phi}) \leftrightarrow S_{I_\phi}((x_i)_{i\in I_\phi}))$, and
            \item $\forall(x_i)_{i\in I}(S_I((x_i)_{i\in I}) \leftrightarrow \exists(x_i)_{i\in J\setminus I} S_{J}((x_i)_{i\in J}))$,
        \end{itemize}
        where $I \subseteq J\subseteq \kappa$, $I$ and $J$ are finite, and $\phi\in\Gamma$.
    \end{enumerate}
\end{definition}

\noindent The next lemma makes explicit the motivation behind the choice of $\Delta_\Gamma$.

\begin{lemma}\label{intuition}
    Let $\M$ be a $\tau_T$-structure. For $\phi\in\Gamma$ and $I\subseteq\kappa$, we denote
    \begin{align*}
            X^\M_\phi &= \{s\colon \{x_i \mid i\in I_\phi\} \to \M \mid (s(x_i))_{i\in I_\phi}\in R_\phi^\M \} \text{ and} \\
            Y^\M_I &= \{s\colon \{x_i \mid i\in I\}\to\M \mid (s(x_i))_{i\in I}\in S_I^\M\}.
        \end{align*}
    Then $\M\models\Delta_\Gamma$ if and only if the following hold.
    \begin{enumerate}
        \item $X^\M_\phi \neq \emptyset$ for all $\phi\in\Gamma$.
        \item $X^\M_\phi = Y^\M_{I_\phi}$ for all $\phi\in\Gamma$.
        \item For all finite $I,J\subseteq\kappa$, if $I\subseteq J$, then $Y^\M_I = Y^\M_J\restriction I$.
    \end{enumerate}
\end{lemma}
\begin{proof} 
    We associate the properties above with the schemas of formulas from $\Delta_\Gamma$. First,
    \begin{align*}
        \M\models\exists \vec{v} R_\phi(\vec{v}) &\iff R_\phi^\M\neq\emptyset \iff X_\phi^\M\neq\emptyset.
    \end{align*}
    Second,
    \begin{align*}
        \M\models\forall(x_i)_{i\in I_\phi} (R_\phi((x_i)_{i\in I\phi}) \leftrightarrow S_{I_\phi}((x_i)_{i\in I_\phi})) &\iff R_\phi^\M = S_{I_\phi}^\M \iff X_\phi^\M = Y^\M_{I_\phi}.
    \end{align*}
    Third, let $I,J\subseteq\kappa$ be finite, $I\subseteq J$. Then
    \begin{align*}
        &\M\models\forall(x_i)_{i\in I}(S_I((x_i)_{i\in I}) \leftrightarrow \exists(x_i)_{i\in J\setminus I} S_{J}((x_i)_{i\in J})) \\
        \iff {}& S_I^\M = \{(a_i)_{i\in I}\in\M^I \mid \exists (a_i)_{i\in J\setminus I}\ (a_i)_{i\in J}\in S_J^\M\} \\
        \iff {}& Y_I^\M = \{s\colon\{x_i \mid i\in I\}\to\M \mid \exists s'\in Y_J^\M\ (s'\restriction\{x_i \mid i\in I\} = s)\} \\
        \iff {}& Y_I^\M = Y_J^\M\restriction\{x_i \mid i\in I\},
    \end{align*}
    \noindent proving our claim.
\end{proof}

Informally, the above lemma states that if $\M\models\Delta_\Gamma$, then there is a coherent directed system of teams $(Y^\M_I)_{I\subseteq \kappa, |I|<\omega}$ such that when restricted to the free variables of any $\phi\in\Gamma$, $Y^\M_I[\Fv(\phi)]$ is a relation that we would like to satisfy the $\eso$-translation of $\phi$. Hence $\Delta_\Gamma$ allows us to attempt to merge these teams into a single team with domain $\{x_i \mid i<\kappa\}$ that would satisfy the whole of $\Gamma$.

The proof of compactness using the translation into $\eso$ consists of two steps. The first step is finding a single model $\M$ and, for every finite subset $\Gamma_0$ of $\Gamma$, a team $X_0$ such that $\M\models_{X_0}\Gamma_0$. The second step consists in merging all the teams satisfying the finite subsets of $\Gamma$ to find one team satisfying $\Gamma$ itself. The following lemma provides us with the first step and follows exactly as in \cite{https://doi.org/10.48550/arxiv.1904.08695}.

\begin{lemma}[Kontinen, Yang]\label{yang.kontinen.construction}
	Let $\Gamma$ be a finitely satisfiable set of formulas of  $\mathcal{L}$. Then there is a structure $\M$ in the expanded vocabulary $\tau_\Gamma$ such that $\M\models \chi_\phi(R_\phi)$ for all $\phi\in\Gamma$ and, additionally, $\M\models\Delta_\Gamma$.
\end{lemma}
\begin{proof}
	Let  $\Gamma'= \{\alpha_\phi \mid \phi \in \Gamma \}$, where $\alpha_\phi$ is as in \cref{notation}. Let $\Gamma'_0$ be a finite subset of $\Gamma'$. Then $\Gamma_0 \coloneqq \{ \phi \mid \alpha_\phi \in \Gamma'_0  \}$ is a finite subset of $\Gamma$, whence by assumption there are a model $\mathcal{N}$ and a non-empty team $Y$ such that $\mathcal{N}\models_Y \Gamma_0$. By locality we may assume that $\dom(Y)=\{x_i \mid i<\kappa\}$. Then, since  $\mathcal{L}$ satisfies a version of Theorem~\ref{translation}, $(\N,R_\phi^\N)\models\chi_\phi(R_\phi)$ for every $\phi\in \Gamma_0$, where $R_\phi^\N = Y[\Fv(\phi)]$. Now, for every $\phi\in\Gamma_0$, there is a tuple $\mathcal{R}_\phi^\N \coloneqq ((R^0_\phi)^\N,\dots,(R^{n_\phi}_\phi)^\N)$ of relations such that $(\N, R_\phi^\N,\mathcal{R}_\phi^\N)\models \alpha_\phi$. Then let
	\[
	    \N'= (\N, (R_\phi^{\N})_{\phi\in\Gamma}, (\mathcal{R}_\psi^\N)_{\psi\in\Gamma_0}, (S_I^\N)_{I\subseteq \kappa, |I|<\omega}),
	\]
	where $S_I^\N= Y[(x_i)_{i\in I}]$. It follows that $ \N'\models \alpha_\phi$ for all $\phi \in \Gamma_0$. By Lemma~\ref{intuition}, it is also clear that $\N'\models\Delta_\Gamma$.

	Thus we have that $\N'\models \Gamma'_0\cup\Delta_\Gamma$. Therefore, since $\Gamma'_0$ was an arbitrary finite subset of $\Gamma'$, it follows by the compactness theorem of first-order logic that there is a model $\M'\models \Gamma'\cup\Delta_\Gamma$. Finally, let $\M$ be the $\tau_\Gamma$-reduct of $\M'$. Then $\M\models \Delta_\Gamma$ and, since for all $\phi\in\Gamma$ we have $\mathcal{M}'\models \alpha_\phi$, it also follows that  $\mathcal{M}\models \chi_\phi(R_\phi)$.
\end{proof}

We now move to the second step of the compactness proof. In  \cite{https://doi.org/10.48550/arxiv.1904.08695}, Kontinen and Yang achieve it by using the fact that $\M\models\Delta_\Gamma$ to merge all the relations $R_\phi^\M$ into a unique team. However, when assuming the uncountability of $\Fv(\Gamma)$ one cannot proceed in the same fashion, and first needs to consider a suitably saturated elementary extension $\M'$ of $\M$. It can then be shown that, in such extension $\M'$, each $R_\phi^\M$ can be merged together to obtain the team that we need. 

We recall some definitions from model theory.
\begin{definition}
    Let $\M$ be a $\tau$-structure.
    \begin{enumerate}
        \item Let $A\subseteq\M$. An \emph{$n$-type} over $A$ in variables $v_0,\dots,v_{n-1}$ is any nonempty set $p$ of first-order $\tau\cup\{a \mid a\in A\}$-formulas $\phi(v_0,\dots,v_{n-1})$, where an element $a\in A$ is thought of as a constant symbol whose interpretation in $\M$ is the element itself.
        
        \item We say that a type $p(\vec{v})$ is \emph{realised} in $\M$ if there is $\vec{b}\in\M^n$ such that $\M\models\phi(\vec{b})$ for all $\phi(\vec{v})\in p$.
        
        \item We say that a type $p$ is \emph{consistent} if each of its finite subtypes is realised in $\M$.
        
        \item Let $\kappa$ be a cardinal number. We say that a model $\M$ is \emph{$\kappa$-saturated} if for every $A\subseteq\M$ such that $|A|<\kappa$, every consistent type $p$ over $A$ is realised in $\M$.
        
        \item A $\tau$-structure $\N$ is an \emph{elementary extension} of $\M$ if $\M$ is a substructure of $\N$ and for all first-order formulas $\phi(\vec{v})$ and $\vec{a}\in\M$, we have
        \[
            \M\models\phi(\vec{a}) \iff \N\models\phi(\vec{a}).
        \]
    \end{enumerate}
\end{definition}

For the proof of the following fact, see e.g.~\cite{Chang-Keisler}.
\begin{fact}
    Let $\M$ be a $\tau$-structure, and let $\kappa$ be a regular cardinal such that $\kappa\geq|\tau|$. Then $\M$ has a $\kappa$-saturated elementary extension.
\end{fact}

\begin{theorem}[Compactness]
	Every  finitely satisfiable set $\Gamma$ of formulas of  $\mathcal{L}$ is satisfiable.
\end{theorem}
\begin{proof}
	Let $x_i$, $i<\kappa$, enumerate the free variables of $\Gamma$. By \cref{yang.kontinen.construction}, we obtain a structure $\M_0$ in the vocabulary $\tau_\Gamma$ such that $\M_0\models\chi_\phi$ for all $\phi\in\Gamma$ and $\M_0\models\Delta_\Gamma$. Let $\N_0$ be the expansion of $\M_0$ with interpretation for $\mathcal{R}_\phi^{\N_0}= ((R^0_\phi)^{\N_0},\dots,(R^{n_\phi}_\phi)^{\N_0})$, for every $\phi \in \Gamma$. Then $\N_0\models\alpha_\phi$ for every $\phi \in \Gamma$.

    Let $\xi\geq\kappa$ be a regular cardinal, and let $\N$ be a $\xi$-saturated elementary extension of $\N_0$. By elementarity, $\N\models\Delta_\Gamma$ and $\N\models\alpha_\phi$ for all $\phi\in\Gamma$. Let $\M$ be the $\tau_\Gamma$-reduct of $\N$. Then we have $\M\models\Delta_\Gamma$ and $\M\models\chi_\phi(R_\phi)$ for all $\phi\in\Gamma$, and $\M$ is $\xi$-saturated.
	
    To show that $\Gamma$ is satisfiable, we need to merge the several finitary relations $R_\phi^\M$ into a unique infinitary relation. We define the infinitary relation
	\[
	   \S:= \{ (a_i)_{i<\kappa}\in\M^\kappa \mid \text{$\M\models S_I((a_i)_{i\in I})$ for every finite $I\subseteq\kappa$} \}.
	\]
	We then let $Y = \{s\colon\{x_i \mid i<\kappa\}\to\M \mid (s(x_i))_{i<\kappa}\in\S\}$. For all $\phi\in\Gamma$ and finite $I\subseteq\kappa$, we let $X_\phi^\M$ and $Y_I^\M$ be as in \cref{intuition}. By the version of \cref{translation} for the logic $\mathcal{L}$, $\M\models_{X_\phi^\M}\phi$ for all $\phi\in \Gamma$. Since by \cref{intuition} we have $Y_{I_\phi}^\M = X_\phi^\M$ and $X_\phi^\M\neq\emptyset$ for any $\phi\in\Gamma$, we obtain
	\begin{equation}\label{projections.are.as.desired}
	    \M\models_{Y_{I_\phi}^\M}\phi \quad\text{and}\quad Y_{I_\phi}^\M\neq\emptyset.
	\end{equation}
	Now it suffices to show that for any finite $I\subseteq\kappa$, we have $Y\restriction\{x_i \mid i\in I\} = Y_I^\M$, as then
	\begin{itemize}
	    \item in particular, $Y\restriction\Fv(\phi) = Y_{I_\phi}^\M$ for all $\phi\in\Gamma$, whence, by~\eqref{projections.are.as.desired} and locality, $\M\models_Y\phi$ for all $\phi\in\Gamma$, i.e. $\M\models_Y\Gamma$,
	    \item $Y$ is nonempty, as otherwise the restrictions $Y_{I_\phi}^\M$ would be empty, which by~\eqref{projections.are.as.desired} they are not.
	\end{itemize}
	We proceed to prove that $Y\restriction\{x_i \mid i\in I\} = Y_I^\M$ for all finite $I\subseteq\kappa$. For this, fix a finite $I\subseteq\kappa$.
	\begin{itemize}
	    \item[$(\subseteq)$] Let $s\in Y\restriction\{x_i \mid i\in I\}$. Then there is $s'\in Y$ such that $s = s'\restriction\{x_i \mid i\in I\}$. By the definition of $Y$, $(s'(x_i))_{i<\kappa}\in\S$. Now by the definition of $\S$, we have $\M\models S_I((s'(x_i))_{i\in I})$, i.e. $(s'(x_i))_{i\in I}\in S_I^\M$. But as $s(x_i)=s'(x_i)$ for all $i\in I$, we have $(s(x_i))_{i\in I}\in S_I^\M$. By the definition of $Y_I^\M$, this means that $s\in Y_I^\M$.
	    
	    \item[$(\supseteq)$] We fix $s\in Y_I^\M$ and let $a_i = s(x_i)$ for all $i\in I$. To show that $s\in Y\restriction\{x_i \mid i\in I\}$, it is enough to find $b_i$, $i<\kappa$, such that $(b_i)_{i<\kappa}\in\S$ and $b_i = a_i$ for $i\in I$. To simplify notation, we assume without loss of generality that $I$ is a (finite) initial segment (and thus an element) of $\kappa$, and we denote $\vec{a} = (a_i)_{i\in I}$. We define $b_i$, $I\leq i<\kappa$, recursively as follows. When $b_j$, $j<i$, have been defined, define the $1$-type
	    \[
	        p_i \coloneqq \{S_{I\cup\{j_0,\dots,j_{n-1},i\}}(\vec{a},b_{j_0},\dots,b_{j_{n-1}},v_0) \mid n<\omega,\ I\leq j_0<\dots<j_{n-1}<i \}
	    \]
	    over the set of parameters $\{b_j \mid j<i\}$, and let $b_i$ be an element of $\M$ that realises $p_i$.

	    We show by induction on $i$ that the element $b_i$ that realises $p_i$ always exists. If $i=I$, then $p_i = \{S_{I\cup\{i\}}(\vec{a},v_0)\}$. As $\M\models\Delta_\Gamma$, in particular
	    \[
	        \M\models\forall(x_j)_{j\in I}(S_I((x_j)_{j\in I}) \leftrightarrow \exists x_i S_{I\cup\{i\}}((x_j)_{j\in I\cup\{i\}})),
	    \]
	    and since $\M\models S_I(\vec{a})$, we then obtain $\M\models\exists x_i S_{I\cup\{i\}}(\vec{a},x_i)$. Hence we find $b_i\in\M$ such that $b_i\models p_i$.
	    
	    Then suppose that $i>I$ and by the induction hypothesis $b_j\models p_j$ for $I\leq j<i$. Then, as the set $\{b_j \mid j<i\}$ of parameters has power $|i|<\kappa\leq\xi$ and $\M$ is $\xi$-saturated, $p_i$ is realised in $\M$ as long as it is consistent. So left is to show that $p_i$ is consistent.
	    
	    Let $q\subseteq p_i$ be finite. Now $q = \{S_{I\cup J_m\cup\{i\}}(\vec{a},(b_j)_{j\in J_m},v_0) \mid m<k\}$ for some finite $J_0,\dots,J_{k-1}\subseteq i\setminus I$. Let $J=\bigcup_{m<k}J_m$. If $J$ is empty then $q\subseteq p_I$ and $q$ is realised by $b_I$, so we may assume $J\neq\emptyset$. Let $j_0,\dots,j_l$ enumerate $J$ in increasing order. Now, as $J\subseteq i$, we have $j_l<i$, and hence by the induction hypothesis, $b_{j_l}\models p_{j_l}$. Notice that $S_{I\cup J}(\vec{a},b_{j_0},\dots,b_{j_{l-1}},v_0)\in p_{j_l}$, and hence $\M\models S_{I\cup J}(\vec{a},b_{j_0},\dots,b_{j_l})$. Since $\M\models\Delta_\Gamma$, we have
	    \[
	        \M\models\forall(x_j)_{j\in I\cup J}[S_{I\cup J}((x_j)_{j\in I},(x_j)_{j\in J}) \leftrightarrow \exists x_i S_{I\cup J\cup\{i\}}((x_j)_{j\in I},(x_j)_{j\in J},x_i)],
	    \]
	    whence $\M\models\exists x_i S_{I\cup J\cup\{i\}}(\vec{a},b_{j_0},\dots,b_{j_l},x_i)$. It then follows that,  for some $c\in\M$, $\M\models S_{I\cup J\cup\{i\}}(\vec{a},b_{j_0},\dots,b_{j_l},c)$. Then, for any $m<k$ we have
	    \[
	        \M\models \exists(x_j)_{j\in J\setminus J_m}S_{I\cup J\cup\{i\}}(\vec{a},(x_j)_{j\in J},c)[b_j/x_j]_{j\in J_m}.
	    \]
	     Moreover, since $\M\models\Delta_\Gamma$, in particular for any $m<k$, $\M$ satisfies the sentence
	    \begin{align*}
	       \forall(x_j)_{j\in I\cup J_m\cup\{i\}}  & \; [S_{I\cup J_m\cup \{i\}}((x_j)_{j\in I},(x_j)_{j\in J_m},x_i) \\ & \; \leftrightarrow \exists(x_j)_{j\in J\setminus J_m}S_{I\cup J\cup\{i\}}((x_j)_{j\in I},(x_j)_{j\in J},x_i)].
	    \end{align*}
	    We then obtain $\M\models S_{I\cup J_m\cup\{i\}}(\vec{a},(b_j)_{j\in J_m},c)$, showing that $c\models q$. This finishes the proof. \qedhere
	\end{itemize}
\end{proof}

\section{Conclusion and Future Directions}

In this article we provided two proofs of the compactness theorem  for several extensions of first-order logic over team semantics, including in particular independence logic and its fragments.  In fact, although compactness for sets of sentences of (in)dependence logic had already been  studied in the literature, compactness for sets of formulas was considered only recently by \textcite{https://doi.org/10.48550/arxiv.1904.08695}, who proved compactness for sets of formulas with countably many variables. In this work we strengthened this result and showed that several logics over team semantics satisfy the compactness theorem with respect to arbitrary sets of formulas. In \cref{compactness-by-ultraproducts} we built upon  Lück's previous work  \cite{luck2020team} and we proved compactness by providing a suitable version of  Łoś' Theorem. On the other hand, in \cref{compactness-by-translation}, we used saturated models to generalize Kontinen's and Yang's proof to arbitrary set of formulas. 

We conclude by noticing that,  while team semantics has been largely studied from the point of view of finite model theory, there has not been an extensive study of the (infinite) model theory of teams. Given the central role played by compactness in elementary model theory, it is natural to inquire how much of model theory could be successfully replicated in the context of team semantics. In particular, we believe that one possibly fruitful direction could be to define a suitable notion of type (of a team, instead of an element or tuple) and prove that there is a (notion of) monster model. One could also search for a natural definition of Galois types for teams and adapt the framework of abstract elementary classes to the setting of team semantics. At the same time, it seems also important to provide examples of theories which could be studied under the light of team semantics. We leave these and other questions to future research.

\printbibliography

\end{document}